\documentclass[11pt,a4paper,reqno]{amsart}
\usepackage{epsfig}
\usepackage{color}
\usepackage{amssymb,amsmath,amsthm,amstext,amsfonts}
\usepackage[normalem]{ulem}
\usepackage{amsmath,amscd,amstext,amsthm,amsfonts,latexsym}
\usepackage[colorlinks=true,linkcolor=blue, urlcolor=red, citecolor=blue, backref]{hyperref}	

\usepackage{amsmath,amssymb,amsthm}
\usepackage{color}
\usepackage[active]{srcltx}
\usepackage{a4wide}
\usepackage{amssymb, amsmath}
\usepackage{graphicx}
\usepackage{float}
\usepackage[active]{srcltx}

\def \cP {\mathcal{P}}

\def \vep {\varepsilon}

\def \supp {\mathrm{supp } }

\def \cW {\mathcal{W}}

\theoremstyle{plain}

\newtheorem{theorem}{Theorem}[section]
\newtheorem{lemma}[theorem]{Lemma}
\newtheorem{proposition}[theorem]{Proposition}
\newtheorem{corollary}[theorem]{Corollary}

\newtheorem{maintheorem}{Theorem}

\newtheorem{definition}[theorem]{Definition}

\newtheorem{remark}[theorem]{Remark}

\numberwithin{equation}{section}

\newcommand{\intav}[1]{\mathchoice {\mathop{\vrule width 6pt height 3 pt depth  -2.5pt
\kern -8pt \intop}\nolimits_{\kern -6pt#1}} {\mathop{\vrule width
5pt height 3  pt depth -2.6pt \kern -6pt \intop}\nolimits_{#1}}
{\mathop{\vrule width 5pt height 3 pt depth -2.6pt \kern -6pt
\intop}\nolimits_{#1}} {\mathop{\vrule width 5pt height 3 pt depth
-2.6pt \kern -6pt \intop}\nolimits_{#1}}}

\title[Topological and metric emergence of continuous maps]{Topological and metric emergence of continuous maps}

\begin{document}

\author[M. Carvalho]{Maria Carvalho}
\address{CMUP \& Departamento de Matem\'atica, Faculdade de Ci\^encias da Universidade do Porto, Rua do Campo Alegre s/n, 4169--007 Porto, Portugal}
\email{mpcarval@fc.up.pt}

\author[F. Rodrigues]{Fagner B. Rodrigues}
\address{Departamento de Matem\'atica, Universidade Federal do Rio Grande do Sul, Brazil.}
\email{fagnerbernardini@gmail.com}

\author[P.Varandas]{Paulo Varandas}
\address{CMUP, Faculdade de Ci\^encias da Universidade do Porto \& Departamento de Matem\'atica e Estat\'istica, Universidade Federal da Bahia, Brazil.}
\email{paulo.varandas@ufba.br}

\date{\today}

\begin{abstract}
We prove that every homeomorphism of a compact manifold with dimension one has zero topological emergence, whereas in dimension greater than one the topological emergence of a $C^0$-generic homeomorphism is maximal, equal to the dimension of the manifold. We also show that the metric emergence of continuous self-maps on compact metric spaces has the intermediate value property.
\end{abstract}

\maketitle

\tiny
\tableofcontents
\normalsize

\section{Introduction}

The topological entropy is an invariant by topological conjugation which quantifies to what extend nearby orbits diverge as the dynamical system evolves.
On a compact metric space, a Lipschitz map has finite topological entropy. However, if the dynamics is just continuous, the topological entropy may be infinite. Actually, K. Yano proved in \cite{Yano} that, on compact smooth manifolds with dimension greater than one, the set of homeomorphisms having infinite topological entropy are $C^0$-generic. So the topological entropy is not an effective label to classify them. Bringing together dimension and dynamics, E. Lindenstrauss and B. Weiss \cite{LW2000} introduced the notion of upper metric mean dimension of a continuous self-map $f$ of a compact metric space $(X,d)$, which may be thought as a mean upper box-counting dimension.  Its value is metric dependent and always upper bounded by the upper box dimension of the space $X$, defined by
\begin{align*}
\mathrm{\overline{dim}_B}\, X \,=\, \limsup_{\varepsilon\,\to\, 0^+}\,\frac{\log S_X(\varepsilon)}{-\log \varepsilon}
\end{align*}
where $S_X(\varepsilon)$ is the maximum cardinality of an $\vep$-separated subset of $X$ (see \cite{Falconer,Pesin} for more details). Thus, it is natural to ask what is the upper metric mean dimension of a $C^0$-generic homeomorphism of $X$, and whether there exists a homeomorphism of $X$ having a prescribed value in the interval $[0, \mathrm{\overline{dim}_B}\, X]$ as its upper metric mean dimension. These questions were partially answered in \cite{CRV}, where we proved that there exists a $C^0$-Baire generic subset of homeomorphisms of any compact smooth manifold with dimension $\mathrm{\dim} X \geqslant 2$ whose elements have the highest possible upper metric mean dimension, namely $\mathrm{\dim} X$; and that any level set of the metric mean dimension of continuous interval self-maps is $C^0$-dense.

\smallskip

Recently, Bochi and Berger introduced in \cite{BB} another concept to quantify the statistical complexity of a system: the topological emergence of a continuous self-map of a compact metric space $X$, which evaluates the size of the space of Borel $f$-invariant and ergodic probability measures (cf. Subsection~\ref{sec:top-emergence} for the definition and more details). To illustrate its importance, they proved, among other equally interesting general results for diffeomorphisms on surfaces, that withi $C^{1+\alpha}$ conformal expanding maps admitting a hyperbolic basic set $\Lambda$ the topological emergence is the largest possible, that is, equal to the upper box dimension of $\Lambda$.
This means that, when $\mathrm{\overline{dim}_B}\, \Lambda > 0$, the number of $\vep$-distinguished ergodic probability measures grows super-exponentially with respect to the parameter $\vep$. Our first aim in this work is to characterize the topological emergence of $C^0$-generic homeomorphisms acting on compact manifolds.

\subsection{Topological emergence}\label{sec:top-emergence}

We start by recalling the concept of topological emergence which measures the complexity of the space of ergodic probability measures preserved by a map. Given a compact metric space $X$ and a continuous map $f:X \to X$, we denote by $\mathfrak{B}$ the $\sigma$-algebra of the Borel subsets of $X$, by $\mathcal{M}_1(X)$ the space of Borel probability measures on $X$, by $\mathcal{M}_f(X)$ its subset of $f$-invariant elements, and by $\mathcal{M}_f^{\mathrm{erg}}(X)$ the subset of $f$-invariant and ergodic probability measures.

\begin{definition}
Let $X$ be a compact metric space, $f:X \to X$ be a continuous map and $\mathcal{D}$ be a distance on the space $\mathcal{M}_1(X)$ such that $(\mathcal{M}_1(X),\mathcal{D})$ is compact. The \emph{topological emergence map} associated to $f$ is the function
$$\vep \in \,\,]0, +\infty[ \quad \mapsto \quad \mathcal E_{\mathrm{top}}(f)(\vep)$$
where $\mathcal E_{\mathrm{top}}(f)(\vep)$ denotes the minimal number of balls of radius $\vep$ in $(\mathcal{M}_1(X),\mathcal{D})$ necessary to cover the set $\mathcal{M}_f^{\mathrm{erg}}(X)$.
\end{definition}

It is clear from the previous definition that the topological emergence depends on the metric we consider in $\mathcal{M}_1(X)$. In what follows, we will always assume that $\mathcal{D}$ is one of the Wasserstein metrics $W_p$, for some $p \geqslant 1$, or the L\'evy-Prokhorov metric $\mathrm{LP}$ (both metrics are defined in Subsection~\ref{subsec:metric}). These metrics induce in $\mathcal{M}_1(X)$ the weak$^*$-topology (cf. \cite{St}).

\begin{definition}
The \emph{upper and lower metric orders} of a compact metric space $(Y,D)$, defined by Kolmogorov and Tikhomirov \cite{Kolmogorov-Tikhomirov}, are given respectively by
\begin{align*}
\overline{\mathrm{mo}}\,(Y) = \limsup_{\varepsilon\,\to\, 0^+}\,\frac{\log\log S_Y(\varepsilon)}{-\log \varepsilon}
\qquad\text{and}\qquad
\underline{\mathrm{mo}}\,(Y) = \liminf_{\varepsilon\,\to\, 0^+}\,\frac{\log\log S_Y(\varepsilon)}{-\log \varepsilon}
\end{align*}
where $S_Y(\varepsilon)$ denotes the maximal cardinality of an $\vep$-separated subset of $Y$. In case both quantities coincide we simply denote them by ${\mathrm{mo}}\,(Y)$, the metric order of the set $Y$. This notions may be extended in a straightforward way to nonempty subsets of $Y$.
\end{definition}

\smallskip

To define the next concept, we need to select either a Wasserstein metric or the L\'evy-Prokhorov metric, but its value does not depend on this choice (cf. \cite{BB}).

\begin{definition}
The \emph{topological emergence} of a continuous self-map $f$ on a compact metric space $X$, which we will denote by $\mathcal E_{\mathrm{top}}(f)$, is the upper metric order of the space of Borel $f$-invariant ergodic probability measures on $X$ endowed with either the Wasserstein metric $W_p$, for some $p \in [1, +\infty[$, or the L\'evy-Prokhorov metric (we denote by $\mathrm{LP}$).
\end{definition}

In what follows we will consider $\log \log 1 = 0$. Therefore, a uniquely ergodic map $f$ has zero topological emergence.

\smallskip

Berger and Bochi proved in \cite[Theorem 1.3]{BB}) that, if $f$ is a continuous map acting on a compact metric space $X$ whose upper and lower box dimensions are $\mathrm{\overline{dim}_B}\, X$ and $\mathrm{\underline{dim}_B}\, X$, respectively, then for any $p \geqslant 1$ one has
\begin{align}\label{ineq:metric_order-box_dim}
\mathrm{\underline{dim}_B}\, X \,\leqslant\, \underline{\mathrm{mo}}\,(\mathcal M_1(X),\mathrm W_p) \,\leqslant\, \overline{\mathrm{mo}}\,(\mathcal M_1(X),\mathrm W_p) \,\leqslant\, \mathrm{\overline{dim}_B}\, X
\end{align}
and that similar inequalities hold if we consider $\mathcal M_1(X)$ endowed with the distance $\mathrm{LP}$. In particular, this ensures that
\begin{align}\label{eq:seq-ineq}
\limsup_{\varepsilon\,\to\,0^+}\,\frac{\log \log\mathcal E_{\mathrm{top}}(f)(\varepsilon)}{-\log\varepsilon}
   \,=\, \overline{ \mathrm{mo}}\,(\mathcal M_f^{\mathrm{erg}}(X),\mathrm W_p) \,\leqslant\, \overline{ \mathrm{mo}}\,(\mathcal M_1(X),\mathrm W_p) \,\leqslant \, \mathrm{\overline{dim}_B}\, X.
\end{align}

\subsection{Metric emergence}\label{sec:metric-emergence}

Fix a compact metric space $(X,d)$, a continuous map $f: X \to X$, a positive integer $n$ and $x \in X$. The $n^{\text{th}}$-\emph{empirical measure associated to} $x$ is defined by 
$$\mathrm{e}^f_n(x) = \frac{1}{n}\,\sum_{i=0}^{n-1}\,\delta_{f^i(x)}$$
where $\delta_z$ denotes the one-point Dirac probability supported on $z$. We recall that, if $\mu$ is an $f$-invariant probability measure, then the Birkhoff's ergodic theorem guarantees that for $\mu$-almost every $x\in X$ the sequence $\big(\mathrm{e}^f_n(x)\big)_{n\,\in\,\mathbb N}$ converges in the weak$^*$-topology to a unique probability measure (cf. \cite{Wa}), which we denote by $\mathrm{e}^f(x)$ and call \emph{empirical measure associated to $x$ by $f$}. For instance, given a periodic point $P$ of period $k$, its orbit supports a unique invariant probability measure, so called \emph{periodic Dirac measure}, defined by $\mu_P = \frac{1}{k}\, \sum_{i=0}^{k-1} \, \delta_{f^i(P)}$, which coincides with $\mathrm{e}^f(P)$. Misiurewicz gives in \cite{Mis} an example of a homeomorphism $f\colon \mathbb{T}^2 \to \mathbb{T}^2$ of the $2$-torus that is expansive, has the specification property and such that, for Lebesgue almost every point $x \in \mathbb{T}^2$, the sequence $\big(\mathrm{e}^f_n(x)\big)_{n\,\in\,\mathbb N}$ accumulates on the whole $\mathcal M_f(\mathbb{T}^2)$, which in this example is very large.

\begin{definition}\label{def:metric_emergence}
Let $(X,d)$ be a compact metric space, $f$ be a continuous self-map of $X$ and $\mu$ be a probability measure on $X$ (not necessarily $f$-invariant). The \emph{metric emergence map} of $\mu$ assigns to each $\vep>0$ the minimal number $\mathcal{E}_\mu(f)(\varepsilon) = N$ of probability measures $\mu_1,\dots,\mu_N$ such that
\begin{equation}\label{ineq:metric_emergence}
\limsup_{n\,\to\,+\infty}\,\int_{X}\,\min_{1\,\leqslant\, i\,\leqslant\, N} \,\mathcal{D}(\mathrm{e}^f_n(x),\,\mu_i)\; d\mu(x) \leqslant \varepsilon.
\end{equation}
The \emph{metric emergence} of $\mu$ is the limit
$$\mathcal{E}_\mu(f) \,=\, \limsup_{\vep \, \to \, 0^+} \, \frac{\log \log \mathcal{E}_\mu(f)(\vep)}{-\log \vep}.$$
\end{definition}

The previous concepts were introduced in \cite{Berger} when $X$ is a compact manifold and $\mu$ is the Lebesgue measure, and generalized in \cite{BB}. In rough terms, $\mathcal{E}_\mu(f)$ essentially evaluates how non-ergodic the probability measure $\mu$ is. If $\mu$ is $f$-invariant then $(\mathrm{e}^f_n(x))_{n\,\in\,\mathbb N}$ converges to $\mathrm{e}^f(x)$ at $\mu$-almost every $x$, and so \eqref{ineq:metric_emergence} can be replaced by
\begin{equation}\label{ineq:metric_emergence2}
\int_{X}\,\min_{1\,\leqslant\, i\,\leqslant\, N}\, \mathcal{D}(\mathrm{e}^f(x),\,\mu_i)\;d\mu(x) \leqslant \varepsilon.
\end{equation}
Thus, if $\mu$ is $f$-invariant and ergodic then its metric emergence map is minimal, equal to $1$.

\smallskip

By \cite[Proposition 3.14]{BB}, it is known that, if $f:X \to X$ is a continuous map of a compact metric space $X$ and $\mu \in \mathcal{M}_f(X)$, then
$$\mathcal{E}_\mu(f)(\vep) \, \leqslant \, \mathcal{E}_{\mathrm{top}}(f)(\vep) \quad \quad \forall\, \vep > 0$$
provided both emergences are computed using the same $W_p$ or $\mathrm{LP}$ metric on $\mathcal{M}_1(X)$.

\subsection{Main results}

Let $X$ be either $[0,1]$ or $\mathbb S^1$, endowed with the Euclidean metric. Denote by $\mathrm{Homeo}_+(X,d)$ the set of order preserving homeomorphisms of $X$ with the uniform metric $D_{C^0}$ given by
$$D_{C^0}(f,g) = \sup_{x \,\in \,X}\,\big\{d(f(x), g(x)), \, d(f^{-1}(x), g^{-1}(x))\big\}.$$
The set $\mathrm{Homeo}_+(X,d)$ with this distance is a Baire space. Our starting point is the following property of the topological emergence of these homeomorphisms.

\begin{maintheorem}\label{thm:main0} If $X=[0,1]$ or $X=\mathbb S^1$ endowed with the Euclidean metric, then every map in
$\mathrm{Homeo}_+(X,d)$ has zero topological emergence.
\end{maintheorem}

Now let $(X,d)$ be a compact connected smooth manifold $X$ (with or without boundary) of dimension at least two. We will consider both the space $\mathrm{Homeo}(X,d)$ of homeomorphisms on $X$ with the uniform metric $D_{C^0}$
and its subset $\mathrm{Homeo}_\mu(X,d)$ of those homeomorphisms which preserve a Borel probability measure $\mu$ on $X$. For reasons we will explain later, we are mainly interested in $\mathbb{O}\mathbb{U}$-probability measures (so named after the work \cite{O-U} of Oxtoby and Ulam; see also \cite{Akin}), which comply with the following conditions:
\medskip

\noindent $(C_1)\quad$ [Non-atomic] For every $x \in X$ one has $\mu(\{x\})=0$.
\smallskip

\noindent $(C_2) \quad$ [Full support] For every nonempty open set $U \subset X$ one has $\mu(U)>0$.
\smallskip

\noindent $(C_3) \quad$ [Boundary with zero measure] $\mu(\partial X)=0$.

\medskip

\noindent The set of  $\mathbb{O}\mathbb{U}$-probability measures is generic in $\mathcal{M}_1(X)$ (see \cite{DGS}). The next result shows that, contrary to Theorem~\ref{thm:main0}, in a higher dimensional setting the topological emergence of $C^0$-generic conservative homeomorphisms attains its maximum possible value.

\begin{maintheorem}\label{thm:main1}
Let $X$ be a compact smooth manifold with dimension $\mathrm{\dim} X \geqslant 2$, $d$ be a metric compatible with the smooth structure of $X$ and $\mu$ be a $\mathbb{O}\mathbb{U}$-probability measure on $X$. There are $C^0$-Baire generic subsets $\mathfrak R \subset \mathrm{Homeo}(X,d)$ and $\mathfrak R_\mu \subset \mathrm{Homeo}_\mu(X,d)$ such that
$$\mathrm{mo}\,(\mathcal{M}^{\mathrm{erg}}_f(X),\,\mathrm W_p) \,=\, \mathrm{dim} X \quad \quad \forall \,f \in \mathfrak R_\mu$$
and
$$\overline{\mathrm{mo}}\,(\mathcal{M}^{\mathrm{erg}}_f(X),\,\mathrm W_p) \,=\, \mathrm{dim} X \quad \quad \forall \,f \in \mathfrak R.$$
\end{maintheorem}

\smallskip

In the measure preserving setting, Theorem~\ref{thm:main1} is a consequence of the fact that, if $\mu$ is a $\mathbb{O}\mathbb{U}$-probability measure on $X$, then a $C^0$-generic element $f$ in $\mathrm{Homeo}_\mu(X)$ is ergodic (cf. \cite{O-U}), has a dense set of periodic points (cf. \cite{DF}) and the shadowing property \cite{GL}, and therefore satisfies the specification property (cf. \cite{DGS}). This implies that the set of ergodic probability measures
is dense in $\mathcal M_f(X)$, so we are left to show that the metric order of $\mathcal M_f(X)$ is equal to $\mathrm{\dim} X$. This is easier to prove since $\mathcal M_f(X)$ is convex and the existence of pseudo-horseshoes is $C^0$-generic in the conservative context (see \cite{DF} and Section~\ref{sec:perturbative-conservative} for more details).

\smallskip

The proof of Theorem~\ref{thm:main1} for dissipative (that is, non-conservative) homeomorphisms also builds on the construction of pseudo-horseshoes, which were introduced in \cite{Yano} and redesigned in \cite{CRV} to satisfy two conditions: to exist in all sufficiently small scales and to exhibit an adequate separation of sufficiently large sets of points in all steps of their construction. However, as the denseness of ergodic probability measures on the set of the invariant ones is not $C^0$-generic within $\mathrm{Homeo}(X)$ (cf. \cite{Hurley1, DGS, KLP}), the argument in the conservative case is no longer valid. To prove that the topological emergence in this setting is $C^0$-generically maximal we will carry out another upgrade on the construction of the pseudo-horseshoes to guarantee that
we can find sufficiently many ergodic probability measures adequately separated with respect to a Wasserstein metric (see Section~\ref{sec:perturbative-dissipative} for more details).

\smallskip

Given a homeomorphism $f:X \to X$, the map which assigns to each nonempty $f$-invariant compact subset $Z$ of $X$ the topological entropy of the restriction of $f$ to $Z$ fails to satisfy the intermediate value property. See, for instance, the minimal homeomorphism on the $2$-torus with positive entropy presented in \cite{Rees}. A. Katok asked whether the metric entropy map satisfies the intermediate value property. More precisely, Katok conjectured that, for every $C^2$ diffeomorphism $f:X \to X$, acting on a compact connected manifold $X$ with finite topological entropy, and for every $c \in [0, h_{\mathrm{top}}(f)]$, there is a Borel $f$-invariant probability measure $\mu$ such that the metric entropy $h_\mu(f)$ is equal to $c$. This conjecture has been positively answered in a number of contexts (cf. \cite{Sun} and references therein). After Theorem~\ref{thm:main1}, one may also ask if the image of the metric emergence map of a $C^0$-generic $f \in \mathrm{Homeo}_\mu(X)$ is $[0, \mathrm{\overline{dim}_B} X]$. We note that for every continuous self-map $f$ of a compact metric space there exists a Borel $f$-invariant probability measure $\mu$ such that $\mathcal{E}_\mu(f)=\mathcal{E}_{\mathrm{top}}(f)$ (cf. \cite{BB}). Our next result generalizes this assertion, providing a proof of the counterpart of Katok's conjecture for the metric emergence.

\begin{maintheorem}\label{cor:1}
Let $f$ be a continuous self-map on a compact metric space $(X,d)$.
Then for every $0 \leqslant \beta \leqslant \mathcal{E}_{\mathrm{top}}(f)$ there is $\mu\in \mathcal{M}_f(X)$ such that $\mathcal{E}_\mu(f)=\beta$.
\end{maintheorem}

It is known that for $C^0$-generic volume preserving homeomorphisms the Lebesgue measure is ergodic (cf \cite{O-U}), so its metric emergence map is constant and equal to one. On the other hand, by Theorem \ref{thm:main1}, for $C^0$-generic volume preserving homeomorphisms on a compact manifold with dimension at least two, one has $\mathcal{E}_{\mathrm{top}}(f) = \mathrm{\dim} X$. This indicates that, $C^0$-generically in the space of volume preserving homeomorphisms, the probability measure
whose metric emergence attains the maximal value $\mathcal{E}_{\mathrm{top}}(f)$ is not the Lebesgue measure. Yet, in the space of conservative $C^r$ diffeomorphisms, $r \geqslant 1$, in any surface, there exists a $C^r$-open subset for whose generic maps the Lebesgue measure has metric emergence equal to two (cf. \cite[Theorem D]{BB}).

\smallskip

The proof of Theorem~\ref{cor:1} relies on the following intermediate value property for the upper metric order map, which is of independent interest.

\begin{maintheorem}\label{thm2}
Let $(Z, d)$ be a compact metric space. The upper metric order function defined on the space of subsets of $Z$ has the intermediate value property, that is, if $0 \leqslant \beta \leqslant \mathrm{\overline{mo}}\,(Z)$ then there exists a subset $Y_\beta \subset Z$ such that $\mathrm{\overline{mo}}\,(Y_\beta) = \beta$.
\end{maintheorem}

\section{Preliminaries}

For future use, in this section we will recall some definitions and previous results.

\subsection{Metrics on $\mathcal M_1(X)$}\label{subsec:metric}

Given a compact metric space $(X,d)$ it is a well-known fact that the space $\mathcal{M}_1(X)$ of the Borel probability measures on $X$ is compact if endowed with the weak$^*$-topology. Moreover, there are metrics on $\mathcal{M}_1(X)$ inducing this topology, the classic ones being the \emph{Wasserstein distances} and the \emph{L\'evy-Prokhorov distance}. The former are defined by
$$\mathrm{W_p}(\mu,\nu) = \inf_{\pi\,\in\,\Pi(\mu,\nu)}\,\left(\int_{X\times X}\,[d(x,y)]^p\;d\pi(x,y)\right)^{1/p}$$
where $p \in [1, +\infty[$ and $\Pi(\mu,\nu)$ denotes the set of probability measures on the product space $X\times X$ with marginals $\mu$ and $\nu$ (see \cite{Vi} and references therein for more details). The latter is defined by
$$\mathrm{LP}(\mu,\nu) = \inf\,\Big\{\varepsilon > 0 \colon \,\, \forall\, E \in \mathfrak{B} \quad  \forall\, \,\text{$\varepsilon$-neighborhood $V_\varepsilon(E)$ of $E$ one has} $$
$$\quad \quad \quad \quad \quad \quad \quad \quad \quad \nu(E)\leqslant \mu(V_\varepsilon(E))+\varepsilon \quad \text{ and } \quad \mu(E)\leqslant \nu(V_\varepsilon(E))+\varepsilon \Big\}.$$
The reader may find more details about this distance in \cite{Bil}.

\smallskip
Throughout the text we will say that two probability measures on a compact metric space $X$ are $\vep$-apart if their supports distance at least $\vep$ in the Hausdorff metric, that is,
$$\min\,\,\big\{d(x,y): \,x \in \mathrm{supp}(\mu),\; y \in \mathrm{supp}(\nu)\big\} \geqslant \vep.$$
We note that, if $N \in \mathbb{N}$ and $\{x_1,\, x_2, \, \cdots, x_N\}$ is an $\varepsilon$-separated subset of $X$, then the Dirac measures $\delta_{x_1},\,\delta_{x_2}, \cdots, \delta_{x_N}$ are pairwise $\varepsilon$-apart.

\begin{remark}\label{rmk:epsilon-apart-measures}
In \cite[Theorem 1.6]{BB}, the authors proved that, if $\mathcal C\subset \mathcal M_1(X)$ is a convex subset and we denote by $\mathcal A(\mathcal C,\varepsilon)$ the maximal cardinality of pairwise $\varepsilon$-apart probability measures in $\mathcal C$, then
$$\min \,\left\{\inf_{p \,\in \,[1, +\infty[} \, \underline{\mathrm{mo}}\,(\mathcal C, \mathrm W_p), \,\,\,  \underline{\mathrm{mo}}\,(\mathcal C, \mathrm{LP})\right\} \,\geqslant\, \liminf_{\varepsilon\,\to\, 0^+}\,\frac{\log \mathcal A(\mathcal C,\varepsilon)}{-\log\varepsilon}.$$
\end{remark}

\subsection{Pseudo-horseshoes}\label{subsec:pseudohorseshoe}
The main tool to prove our first theorem is a class of compact invariant sets, called pseudo-horseshoes. Such structures were used in \cite{Yano} to prove that $C^0$-generic homeomorphisms, acting on compact manifolds $(X,d)$ with dimension greater than one, have infinite topological entropy; and later in \cite{CRV} to show the existence of a $C^0$-Baire generic subset $\mathfrak{R}_0 \subset \mathrm{Homeo}(X,d)$ where the metric mean dimension is maximal, equal to $\mathrm{\dim} X$. In what follows we recall the main definitions and properties of pseudo-horseshoes on manifolds. We refer the reader to \cite{Hurley1, Hurley2}, where one finds other relevant properties of the atractors and pseudo-horseshoes of generic homeomorphisms.

\subsubsection{Pseudo-horseshoes in $\mathbb R^k$}

Consider in $\mathbb{R}^k$ the norm
$$\|(x_1, \cdots, x_k)\| := \max_{1\,\leqslant \,i\,\leqslant k} \,|x_i|.$$
Given $r > 0$ and $x \in \mathbb{R}^k$, denote $D^k_r(x) = \big\{y \in \mathbb R^k \colon \|x - y\| \leqslant r\big\}$ and $D^k_r = D^k_r\big((0,\dots,0)\big).$ For $1\leqslant j\leqslant k$, let $\pi_j \colon \,\mathbb{R}^k \,\to\,\mathbb{R}^j$ be the projection
on the first $j$ coordinates.

\begin{definition}\label{def:vertical}
Fix $r > 0$, $x=(x_1, \cdots, x_k)$ and $y=(y_1, \cdots, y_k)$ in $\mathbb{R}^k$, and take an open set $U \subset \mathbb R^k$ containing $D_r^k(x)$. Having fixed a positive integer $N$, we say that a homeomorphism $\varphi \colon \,U \,\to\,\mathbb R^k$ has a \emph{pseudo-horseshoe of type $N$ at scale $r$ connecting $x$ to $y$} if the following conditions are satisfied:
\begin{enumerate}
\item $\varphi(x)=y$.
\medskip
\item $\varphi\Big(D_r^k(x)\Big) \subset \mathrm{int} \Big(D_r^{k-1}(\pi_{k-1}(y))\Big) \,\times \,\mathbb R$.
\medskip
\item For $i=0,1,\ldots,\left[\frac{N}{2}\right]$,
$$\varphi\Big(D_r^{k-1}(\pi_{k-1}(x)) \,\times\, \Big\{x_k - r + \frac{4ir}{N}\Big\} \Big) \subset \mathrm{int} \Big(D_r^{k-1}(\pi_{k-1}(y))\Big)\,\times\, (-\infty, \,y_k-r).$$
\item For $i=0,1,\ldots,\left[\frac{N-1}{2}\right]$,
$$\varphi\left(D_r^{k-1}(\pi_{k-1}(x))\,\times\,\Big\{x_k - r + \frac{(4i+2)r}{N}\Big\}\right) \subset \mathrm{int} \Big(D_r^{k-1}(\pi_{k-1}(y))\Big)\,\times\, (y_k + r, \,+\infty).$$
\item For each $i \in \{0,\dots,N-1\}$, the intersection
$$ V_i=D_r^k(y) \,\cap \,\varphi\left(D_r^{k-1}(x) \times \left[x_k - r + \frac{2ir}{N}, \,x_k - r + \frac{(2i+2)r}{N}\right]\right)$$
is connected and satisfies:
\begin{itemize}
\medskip
\item[(a)] $ V_i \,\cap\, (D_r^{k-1}(y) \times \{-r\}) \not=\emptyset$;
\medskip
\item[(b)] $ V_i \,\cap \,(D_r^{k-1}(y)\times \{r\}) \not=\emptyset;$
\medskip
\item[(c)] each connected component of $ V_i \cup \partial D_r^k(y)$ is simply connected.
\end{itemize}
\medskip
Each $V_i$ is called a \emph{vertical strip} of the pseudo-horseshoe.
\end{enumerate}
\end{definition}


\smallskip

\subsubsection{Pseudo-horseshoes in manifolds} So far, pseudo-horseshoes were defined in open sets of $\mathbb{R}^k$. Now we convey this notion to manifolds.

\begin{definition}\label{def:po}
Let $(X,d)$ be a compact smooth manifold of dimension $\mathrm{\dim} X$. Given $f \in \mathrm{Homeo}(X,d)$ and constants $0 < \alpha < 1$, $\delta > 0$, $0 < \vep < \delta$ and $q \in \mathbb{N}$, we say that $f$ has a coherent $(\delta,\,\vep,\,q,\,\alpha)$-\emph{pseudo-horseshoe} if we may find a pairwise disjoint family of open subsets $(\mathcal{U}_i)_{0\,\leqslant\, i \,\leqslant\, q-1}$ of $X$ such that
$$f(\mathcal{U}_{i}) \cap \mathcal{U}_{(i+1)\mathrm{mod} \, q} \neq \emptyset \quad \quad \forall\,\, i$$
and a collection $(\phi_i)_{0\,\leqslant\, i\, \leqslant\, q-1}$ of homeomorphisms
$$\phi_i\colon D_{\delta}^{\mathrm{\dim} X} \subset \mathbb{R}^{\mathrm{\dim} X} \quad \to \quad \mathcal{U}_i \subset X$$
satisfying, for every $0 \leqslant i\leqslant q-1$:
\medskip
\begin{enumerate}
\item $\left(f\circ \phi_i\right)(D_{\delta}^{\mathrm{\dim} X} ) \subset \mathcal{U}_{(i+1)\mathrm{mod} \, q}$.
\medskip
\item The map
$$\psi_i = \phi_{(i+1)\mathrm{mod} \, q}^{-1}\circ f\circ \phi_i \colon \quad D_{\delta}^{\mathrm{\dim} X} \to\, {\mathbb R}^{\mathrm{\dim} X}$$
has a pseudo-horseshoe of type $\lfloor\Big(\frac1\vep\Big)^{\alpha \, \mathrm{\dim} X}\rfloor$ at scale $\delta$ connecting $x=0$ to itself and such that:
\medskip
\begin{enumerate}
\item There are families $\{V_{i,j}\}_{j}$ and $\{H_{i,j}\}_{j}$ of vertical and horizontal strips, respectively, with $j \in \{1, 2, \dots, \lfloor\Big(\frac1\vep\Big)^{\alpha \, \mathrm{\dim} X}\rfloor\}$, such that $H_{i,j} = \psi_i^{-1} \big({V}_{i,j}\big).$
\medskip
\item For every $j_1 \neq j_2 \in \{1, 2, \dots, \lfloor\Big(\frac1\vep\Big)^{\alpha \, \mathrm{\dim} X}\rfloor\}$ we have
$$\min\,\Big\{\inf\,\{\|a-b\|\colon \, a \in V_{i,j_1}, \,b \in V_{i,j_2}\}, \quad \inf\,\{\|z-w\|\colon \, z \in H_{i,j_1}, \,w \in H_{i,j_2}\}\Big\} > \vep.$$
\end{enumerate}
\item For every $0 \leqslant i \leqslant q-1$ and every $j_1 \neq j_2  \in \{1, 2, \dots, \lfloor\Big(\frac1\vep\Big)^{\alpha \, \mathrm{\dim} X}\rfloor\}$, the horizontal strip $H_{i,j_1}$ crosses the vertical strip $V_{(i+1)\mathrm{mod} \, p, j_2}$, where by crossing we mean that there exists a foliation of each horizontal strip $H_{i,j} \subset D_{\delta}^{\mathrm{\dim} X}$ by a family $\mathcal C_{i,j}$ of continuous curves $c\colon [0,1] \to H_{i,j}$ such that $\psi_i(c(0))\in D_{\delta}^{k-1}\times \{-\delta\}$ and $\psi_i(c(1))\in D_{\delta}^{k-1}\times \{\delta\}$.
\end{enumerate}
\end{definition}

Regarding the parameters $(\delta, \vep, q, \alpha)$ that identify the pseudo-horseshoe, we remark that $\delta$ is a small scale determined by the size of the $q$ domains and the charts so that item (1) of Definition~\ref{def:po} holds; $\vep$ is the scale at which a large number
(which is inversely proportional to $\vep$ and involves $\alpha$) of finite orbits is separated, to comply with the demand (2) of Definition~\ref{def:po}; and $\alpha$ is conditioned by the room in the manifold needed to build the convenient amount of $\vep$-separated points.

\begin{definition}\label{def:coherent}
We say that $f$ has a \emph{coherent $(\delta,\vep,p,\alpha)$-pseudo-horseshoe} if the pseudo-horseshoe satisfies the extra condition
\begin{enumerate}
\item[(3)] For every $0 \leqslant i \leqslant p-1$ and every $j_1 \neq j_2  \in \{1, 2, \dots, \lfloor\Big(\frac1\vep\Big)^{\alpha \, \mathrm{dim}\, X}\rfloor\}$, the horizontal strip $H_{i,j_1}$ crosses the vertical strip $V_{(i+1)\mathrm{mod} \, p, j_2}$.
\end{enumerate}
By crossing we mean that there exists a foliation of each horizontal strip $H_{i,j} \subset D_{\delta}^{\mathrm{dim}\, X}$
by a family $\mathcal C_{i,j}$ of continuous curves $c\colon [0,1] \to H_{i,j}$ such that $\psi_i(c(0))\in D_{\delta}^{k-1}\times \{-\delta\}$ and $\psi_i(c(1))\in D_{\delta}^{k-1}\times \{\delta\}$.
\end{definition}

\smallskip

Coherent $(\delta,\vep,q,\alpha)$-pseudo-horseshoes have at least three important features. Firstly, these pseudo-horseshoes persist under $C^0$-small perturbations. Secondly, every homeomorphism which has a coherent $(\delta, \vep, q, \alpha)$-pseudo-horseshoe also has a $(q,\vep)$-separated set with at least $\lfloor\Big(\frac1\vep\Big)^{\alpha\, \mathrm{\dim} X}\rfloor$ elements. The third main property of coherent pseudo-horseshoes is the following proposition. 

\smallskip

Fix a strictly decreasing sequence $(\vep_k)_{k \,\in \,\mathbb{N}}$ in the interval $\,]0,1[$ converging to zero and let $L>0$ be a bi-Lipschitz constant for the charts of a finite atlas of $X$. Denote by $\mathcal O(\vep_k, \alpha)$ the set of homeomorphisms $g \in \mathrm{Homeo}(X,d)$ such that $g$ has a coherent $(\delta, L \vep_k, q, \alpha)$-pseudo-horseshoe, for some $\delta>0$, $q \in \mathbb{N}$ and $L>0$. Then:

\begin{proposition}[\cite{CRV}]\label{C0_generic1}
For every $\alpha \in \,\,]0,1[$ and $k \in \mathbb{N}$, the set $\mathcal O(\vep_k, \alpha)$ is $C^0$-open. Moreover, given $K \in \mathbb{N}$, the union
$$\mathcal O_K(\alpha)\, = \,\bigcup_{\substack{k \,\, \in \,\, \mathbb{N} \\ k \,\, \geqslant\,\, K}}\,\mathcal O(\vep_k, \alpha)$$
is $C^0$-dense in $\mathrm{Homeo}(X,d)$. In particular
$$\mathfrak{R}_0 \, = \, \bigcap_{\alpha\, \in \,\,]0,1[ \,\cap\, \mathbb Q} \; \bigcap_{K \,\in \,\mathbb{N}} \; \bigcup_{\substack{k \,\, \in \,\, \mathbb{N} \\ k \,\, \geqslant\,\, K}}\,\mathcal O(\vep_k, \alpha)$$
is a $C^0$-Baire residual subset of $\mathrm{Homeo}(X,d)$.
\end{proposition}

Regarding the conservative setting, let $\mu$ be an $\mathbb{O}\mathbb{U}$-probability measure on $X$. The methods in \cite{Guih} allow us to make $C^0$-small perturbations of any $\mu$-preserving homeomorphism in order to create coherent pseudo-horseshoes. In particular, the space $\mathcal O_\mu(\vep_k, \alpha)$ of homeomorphisms in $\mathrm{Homeo}_\mu(X,d)$ exhibiting a coherent $(\delta, L\vep_k, q, \alpha)$-pseudo-horseshoe is $C^0$-open and dense in $\mathrm{Homeo}_\mu(X,d)$. Thus, the set
\begin{equation*}\label{def:R}
\mathfrak{R}_1 = \,\bigcap_{\alpha \,\in \,\,]0,1[\, \cap\, \mathbb Q} \,\, \bigcap_{K \,\in\, \mathbb{N}} \,\bigcup_{\substack{k \,\, \in \,\, \mathbb{N} \\ k \,\, \geqslant\,\, K}}\,\mathcal O_\mu(\vep_k, \alpha)
\end{equation*}
is $C^0$-Baire generic in $\mathrm{Homeo}_\mu(X,d)$.

\subsection{Specification property}\label{sec:spec}
According to Bowen~\cite{Bowen}, a continuous map $f: X \to X$ on a compact metric space $(X,d)$ satisfies the \emph{specification property} if for any $\delta>0$ there exists $T(\delta) \in \mathbb{N}$ such that any finite block of iterates by $f$ can be $\delta$-shadowed by an individual orbit provided that the time lag
of each block is larger than the prefixed time $T(\delta)$. More precisely, $f$ satisfies the specification property if for any $\delta>0$ there exists an integer $T(\delta) \in \mathbb{N}$ such that for every $k \in \mathbb{N}$, any points $x_1, \dots, x_k$ in $X$, any sequence of positive integers $n_1, \dots, n_k$ and every choice of integers $T_1, \dots, T_k$ with $T_i \geqslant T(\delta)$, there exists a point $x_0$ in $X$ such that
$$d\Big(f^j(x_0),\,f^j(x_1)\Big) \leqslant \delta \quad \quad \forall \, 0 \leqslant j \leqslant n_1$$
and
$$d\Big(f^{j + n_1 + T_1 + \dots + n_{i-1} + T_{i-1}}(x_0) \;,\; f^j(x_i)\Big) \leqslant \delta \quad \quad \forall \, 2 \leqslant i \leqslant k \quad \forall\, 0\leqslant j\leqslant n_i.$$

\smallskip

It is known that full shifts on finitely many symbols satisfy the specification property; besides, factors of maps with the specification property also enjoy this property (cf. \cite{DGS}).
Moreover, if $\mu$ is a $\mathbb{O}\mathbb{U}$-probability measure, the specification property is $C^0$-Baire generic in $\mathrm{Homeo}_\mu(X,d)$ (cf. \cite{GL}).
\smallskip

The importance of the specification property in the study of the topological emergence is illustrated by the fact that it guarantees the denseness of the set of periodic measures in the space of invariant probability measures (cf. \cite{DGS}), together with the following result, essentially stated by Bochi in \cite{Bseminar}.

\begin{lemma}\label{Bochi_Lemma}
Let $(X,d)$ be a compact metric space and $f:X\to X$ be a continuous map such that $\overline{\mathcal M^{\mathrm{erg}}_f(X)}=\mathcal M_f(X)$. Take a sequence $(\varepsilon_n)_{n \, \in \, \mathbb{N}}$ of positive real numbers satisfying $\lim_{n\,\to\, +\infty}\,\varepsilon_n=0$. Assume that there exist constants $C, L, \gamma > 0$ such that, for every $n \in\mathbb N$, there is an $f$-invariant finite subset $F_n \subset X$ containing only periodic orbits and satisfying the conditions:
\begin{enumerate}
\item[(i)] any two distinct orbits in $F_n$ are uniformly $\varepsilon_n$-separated from each other;
\item[(ii)] the cardinal of $F_n$ is bounded below by $C\, (1/ L\varepsilon_n)^\gamma$, for some constant $C>0$.
\end{enumerate}
Then
$$\limsup_{\varepsilon\,\to\,0^+}\,\frac{\log\;\log\mathcal E_{\mathrm{top}}(f)(\varepsilon)}{-\log\varepsilon} \,\geqslant \gamma.$$
In particular, if $\gamma={\mathrm{\overline{dim}_B}\, X}$ then $\mathcal E_{\mathrm{top}}(f)\,=\, \mathrm{\overline{dim}_B}\, X$.
\end{lemma}

\smallskip

\begin{proof}
Fix $n \in \mathbb N$ and denote by $\# F_n$ the cardinal of $F_n$. Consider the set of ergodic probability measures supported on $F_n$, whose distinct elements are $\varepsilon_n$-apart due to condition (i). Given $\varepsilon>0$, take $N \in\mathbb N$ such that $\varepsilon_n \leqslant \varepsilon$ for every $n \geqslant N$. Then
$$
\mathcal A(\mathcal M_f(X),\varepsilon) \geqslant  \mathcal A(\mathcal M_f(X), \varepsilon_n) \geqslant \#  F_n \quad \quad \forall\, n \geqslant N
$$
where $\mathcal A(\mathcal M_f(X),\varepsilon)$ is the maximal cardinality of pairwise $\varepsilon$-apart probability measures in $\mathcal M_f(X)$. According to Remark~\ref{rmk:epsilon-apart-measures}, these inequalities and the condition (ii) imply that
$$
\underline{\mathrm{mo}}(\mathcal M_f(X),\mathrm W_p) \,\geqslant \, \gamma.
$$
Moreover, by assumption, the closure of $\mathcal M^{\mathrm{erg}}_f(X)$ is $\mathcal M_f(X)$, so
$$
\underline{\mathrm{mo}}(\mathcal M^{\mathrm{erg}}_f(X),\mathrm W_p) \,= \,\underline{\mathrm{mo}}(\mathcal M_f(X),\mathrm W_p)
	\, \geqslant \,\gamma
$$
as claimed. This proves the first statement of the lemma.

\smallskip

In the particular case of $\gamma=\mathrm{\overline{dim}_B}\, X$, we may conclude more, since, as a consequence of \cite[Equation 2.2, Theorem 1.3]{BB}, we know that
$$\overline{\mathrm{mo}}\,(\mathcal M^{\mathrm{erg}}_f(X), \mathrm W_p) \,\leqslant\, \mathrm{\overline{dim}_B}\, X.$$
Thus, $\mathrm{mo}(\mathcal M^{\mathrm{erg}}_f(X),\mathrm W_p) = \mathrm{\overline{dim}_B}\, X$.
\end{proof}

\section{Proof of Theorem~\ref{thm:main0}}

Assume that $X = [0,1]$. Given $f \in \mathrm{Homeo}_+([0,1])$, it is immediate to conclude that the non-wandering set of $f$, say $\Omega(f)$, coincides with the set of fixed points (we denote by $\mathrm{Fix}(f)$). Indeed, each orbit by $f$ is a monotonic bounded sequence, so it converges and, by the continuity of $f$, the limit is a fixed point. In particular, one has
$$\mathcal M^{\mathrm{erg}}_f([0,1]) \,=\, \Big\{\delta_x \colon f(x) = x\Big\}$$
hence $(\mathcal M^{\mathrm{erg}}_f(X),W_p)$ is isometric to a subset of $(X,d)$. This implies that $\mathcal E_{\mathrm{top}}(f)(\varepsilon) = \mathcal O(\vep^{-1})$ for every sufficiently small $\vep>0$, and so $\mathcal E_{\mathrm{top}}(f)= 0$. In particular, for any $f \in \mathrm{Homeo}_+([0,1])$
$$0 \, = \, \mathcal E_{\mathrm{top}}(f) \, = \, \sup\, \{\mathrm{dim}_B(\mu) \colon \, \mu \in \mathcal M^{\mathrm{erg}}_f([0,1])\}.$$

\smallskip

If $X=\mathbb S^1$, consider $f \in \mathrm{Homeo}_+(\mathbb S^1)$ with irrational rotation number $\rho(f)$. As the rotation number map $\rho: \,\mathrm{Homeo}_+(\mathbb S^1) \to [0,1[$ is surjective and continuous, such homeomorphisms are the $C^0$-generic ones. Such an $f$ is uniquely ergodic (cf. \cite{Wa}), so it has zero topological emergence. Moreover, $f$ is topologically conjugate to the irrational rotation $R_{\rho(f)}$, hence the unique ergodic probability measure which $f$ preserves is equivalent to the Lebesgue measure. Therefore, generically in $\mathrm{Homeo}_+(\mathbb S^1)$ one has
$$0 \, = \, \mathcal E_{\mathrm{top}}(f) \, < \, \sup\, \{\mathrm{dim}_B(\mu) \colon \, \mu \in \mathcal M^{\mathrm{erg}}_f(\mathbb S^1)\}.$$

\smallskip

If, otherwise, $f \in \mathrm{Homeo}_+(\mathbb S^1)$ has rational rotation number $\rho(f)$, then there is a conjugation between the restriction of $f$ to its non-wandering set $\Omega(f)$ and the restriction of the rotation $R_{\rho(f)}$ to a closed subset of $S^1$. Thus, every non-wandering point of $f$ is periodic and all the periodic points have the same period (say $m$).  Moreover, $S^1 \setminus \Omega(f)$ is a union of open intervals and each of these intervals is mapped onto itself by the iterate $f^m$ in a fixed-point free manner. In particular, in each of these intervals one has either $f^m(x) < x$ for every $x$ or $f^m(x) > x$ for every $x$. So, the orbit by $f^m$ of any point of each of these open intervals converges to a periodic point of $f$ with period $m$. The proofs of the previous assertions may be found in \cite{Nitecki}.

\smallskip

By Poincar\'e recurrence theorem, the space $\mathcal{M}^{\mathrm{erg}}_f(S^1)$ is determined precisely by the elements of $\Omega(f)$. Consequently, $\mathcal{M}^{\mathrm{erg}}_f(S^1)$ is made up of Dirac measures supported in periodic orbits with equal period $m$; and the $\mathrm{LP}$-distance between them is the distance in $S^1$ between their supports. Thus, given $\vep > 0$, to cover $\mathcal{M}^{\mathrm{erg}}_f(S^1)$ by balls with radius $\vep$ we need at most $\lfloor 1/\vep \rfloor$ of them. Consequently, $\mathcal{E}_{\mathrm{top}}(f)(\vep) \,= \,\mathcal{O}(\vep^{-1})$ for every small $\vep>0$, and so the topological emergence of $f$ is zero. This ends the proof of the theorem.

\section{Proof of Theorem~\ref{thm:main1}: conservative setting}\label{sec:perturbative-conservative}

Let $X$ be a compact smooth manifold with dimension at least two, $d$ be a metric compatible with the smooth structure of $X$ and $\mu$ be a $\mathbb{O}\mathbb{U}$-probability measure on $X$. Denote by $\mathfrak R_s$ the $C^0$-Baire generic subset of $\mathrm{Homeo}_\mu(X,d)$ formed by homeomorphisms which satisfy the specification property (cf. \cite{GL}).

\smallskip

Recall from Subsection~\ref{subsec:pseudohorseshoe} that, given $\alpha \in \,\,]0,1[$, a strictly decreasing sequence $(\vep_k)_{k \,\in \,\mathbb{N}}$ in the interval $\,]0,1[$ converging to zero, a bi-Lipschitz constant $L>0$ for the charts of a finite atlas of $M$ and $k \in \mathbb{N}$, we denote by ${\mathcal O}_\mu(\vep_k, \alpha)$ the set of homeomorphisms $g \in \mathrm{Homeo}_\mu(X,d)$ such that $g$ has a coherent $(\delta, L\vep_k, q, \alpha)$-pseudo-horseshoe, for some $\delta>0$, $q \in \mathbb{N}$ and $L>0$. For every $K \in \mathbb{N}$, define
$${\mathcal O}_{\mu, \, K}(\alpha) \,=\, \bigcup_{\substack{k \,\, \in \,\, \mathbb{N} \\ k \,\, \geqslant\,\, K}}\,{\mathcal O}_\mu(\vep_k, \alpha).$$
The set ${\mathcal O}_{\mu, \, K}(\alpha)$ is $C^0$-open and dense in $\mathrm{Homeo}_\mu(X,d)$ (cf. \cite{Guih}). Thus the intersection
$$\mathfrak{R}_\mu \,=\, \mathfrak{R}_s \; \cap \;\, \Bigg(\bigcap_{\alpha \,\in \,\,]0,1[\, \cap\, \mathbb Q} \,\, \bigcap_{K \,\in\, \mathbb{N}} \,{\mathcal O}_{\mu, \, K}(\alpha) \Bigg)$$
is $C^0$-Baire generic in $\mathrm{Homeo}_\mu(X,d)$.

\smallskip

Given $\alpha \in \,\,]0,1[\, \cap\, \mathbb Q$ and $K \in\mathbb N$, any homeomorphism $g \in \mathfrak{R}_\mu$ has a coherent $(\delta, L\vep_k, q, \alpha)$-pseudo-horseshoe $\Lambda_k$, for some $\delta>0$, $q \in \mathbb{N}$, $L>0$ and $k \geqslant K$. Therefore (cf. \cite[Proposition 6.1]{CRV}), there exists a finite subset $F_K \subset \mathcal M_g(X)$ formed by probability measures supported on $g$-periodic orbits of period $q$ which are $\varepsilon_k$-apart
from each other, and whose cardinality satisfies
$$\# \,F_K \,\geqslant \,\left\lfloor\Big(\frac1{L\vep_k}\Big)^{\alpha\,\mathrm{\dim} X}\right\rfloor.$$
Therefore, by Lemma~\ref{Bochi_Lemma}, the upper metric order of $\mathcal M_g(X)$ is bigger or equal to $\mathrm{\dim} X$, since
$$\mathrm{mo}\,(\mathcal M_g(X), \,\mathrm W_p) \,\geqslant\, \alpha \, \mathrm{\dim} X$$
$\alpha \in \,\,]0,1[\,\, \cap\, \mathbb Q$ is arbitrary. Moreover, the converse inequality
$$\mathrm{mo}(\mathcal M_g(X), \, \mathrm W_p) \,\leqslant \, \mathrm{\dim} X$$
always holds (see \eqref{eq:seq-ineq}). So, $\mathrm{mo}\,(\mathcal M_g(X), \,\mathrm W_p) = \mathrm{\dim} X$.

\smallskip

We are left to deduce from the previous equality that the topological emergence is maximal. As every $g \in \mathfrak R_\mu$ satisfies the specification property, the closure of the space $\mathcal M^{\mathrm{erg}}_g(X)$ is equal to $\mathcal M_g(X)$ (cf. Subsection~\ref{sec:spec}). Thus,
$$\mathrm{mo}\,(\mathcal M^{\mathrm{erg}}_g(X), \, \mathrm W_p) \,= \,\mathrm{mo}\,(\mathcal M_g(X), \,\mathrm W_p) \,=\, \mathrm{\dim} X$$
and similar equalities hold regarding the metric $\mathrm{LP}$. This confirms that
$$\limsup_{\varepsilon\,\to\,0^+}\,\frac{\log\;\log\mathcal E_{\mathrm{top}}(g)(\varepsilon)}{-\log\varepsilon}\, =\, \mathrm{\dim} X \quad \quad \forall\, g \in \mathfrak{R}_\mu$$
and the proof of Theorem~\ref{thm:main1} for $\mathrm{Homeo}_\mu(X,d)$ is complete.

\smallskip

The previous proof yields the following consequence:

\begin{corollary}\label{cor:0} Under the assumptions of Theorem~\ref{thm:main1}, if $f \in \mathfrak R_\mu$ then $\mathcal E_{\mathrm{top}}(f_{|\mathrm{Per}\,f}) \,=\, \mathrm{\dim} X.$
\end{corollary}

\section{Proof of Theorem~\ref{thm:main1}: non-conservative setting}\label{sec:perturbative-dissipative}

The argument in the previous section also shows that
$$\mathrm{mo}(\mathcal M_f(X), \, \mathrm W_p) \,=\, \mathrm{\dim} X \,=\, \mathrm{mo}(\mathcal M_f(X), \, \mathrm{LP}) \quad \quad \forall \, f \in \mathfrak{R}_0$$
where $\mathfrak{R}_0$ is the $C^0$-generic subset of $\mathrm{Homeo}(X,d)$ defined in Proposition~\ref{C0_generic1}. However, the proof we presented for the conservative case does not entirely apply to $\mathrm{Homeo}(X,d)$. Indeed, whereas a $C^0$-generic homeomorphism in
$\mathrm{Homeo}_\mu(X,d)$ is ergodic \cite{O-U}, hence transitive, there exists a $C^0$-open and dense set of homeomorphisms in $\mathrm{Homeo}(X,d)$ which display absorbing regions (cf. \cite[Lemma 3.1]{PPSS} or \cite{Hurley1}),
hence those maps are non-transitive. As transitivity is a necessary condition for the denseness of the ergodic probability measures in the space of invariant ones (cf. \cite{DGS, KLP}), a typical homeomorphism in $\mathrm{Homeo}(X,d)$ does not satisfy the requirements needed to apply Lemma~\ref{Bochi_Lemma}. Actually, such a strategy cannot even be pursued within a coherent pseudo-horseshoe, since an arbitrarily $C^0$-small perturbation of these structures also allows us to create
open trapping regions. Therefore we need to refine the construction of the set $\mathfrak{R}_0$ in order to ensure the existence of an adequate amount of ergodic probability measures at appropriate scales.

\subsection{Topological horseshoes}

We start by establishing a strengthened version of Proposition~\ref{C0_generic1}.

\begin{proposition}\label{C0_generic1-extra} Fix a strictly decreasing sequence $(\varepsilon_k)_{k \in \mathbb{N}}$ in the interval $]0,1[$ such that $\lim_{k \, \to \, +\infty} \varepsilon_k = 0$. For every $\alpha \in \,\,]0,1[$ and $k \in \mathbb{N}$, there exists a $C^0$-open subset $\widehat{\mathcal O}(\vep_k, \alpha)\subset \mathcal O(\vep_k, \alpha)$ such that:
\begin{enumerate}
\item Given $K \in \mathbb{N}$, the union
$$\widehat{\mathcal O}_K(\alpha)\, = \,\bigcup_{\substack{k \,\, \in \,\, \mathbb{N} \\ k \,\, \geqslant\,\, K}}\, \widehat{\mathcal O}(\vep_k, \alpha)$$
is $C^0$-dense in $\mathrm{Homeo}(X,d)$.
\smallskip

\item There exists a constant $C>0$ such that if $h\in \widehat{\mathcal O}(\vep_k, \alpha)$ then:
\smallskip

\begin{itemize}
\item $h$ has a coherent $(\delta,\,L\vep_k,\,q,\,\alpha)$-{pseudo-horseshoe}, for some $\delta>0$, $q \in \mathbb{N}$ and $L>0$.
\smallskip

\item $h$ has a collection of $\lfloor\Big(\frac1{\vep_k}\Big)^{\alpha\, \mathrm{dim}\, X}\rfloor^q$ periodic orbits of period $q$ whose supports are $\vep^q_k$-apart in the Hausdorff metric.
\smallskip


\item There exists a subset $E_h(X) \subset \mathcal M^{\mathrm{erg}}_h(X)$
whose cardinality is larger than
$$C \,\exp\left(\frac1C \Big\lfloor\Big(\frac1{\vep_k}\Big)^{\alpha\, \mathrm{dim}\, X}\Big\rfloor^q\right)$$
and such that any two distinct elements in $E_h(X)$ are $8^{-\frac1p} \vep_k^{q}$-separated in the $W_p$ distance.
\end{itemize}
\end{enumerate}
\end{proposition}

\smallskip

\begin{proof} We start with the construction of the $C^0$-generic set $\mathfrak{R}_0$ described in Proposition~\ref{C0_generic1}. Then, given a homeomorphism $f \in \mathfrak{R}_0$, to overcome the lack of specification within $f$ we will make a small local $C^0$-perturbation of $f$ to obtain a homeomorphism $g$ whose restriction to a fixed arbitrarily small open subset $U$ of $X$ is a $C^1$-diffeomorphism and exhibits in $U$ a horseshoe (that is, a closed invariant set restricted to which the dynamics is conjugate to a full shift on a finite alphabet), where the periodic specification property is valid. Clearly we cannot expect that this horseshoe persists under small $C^0$-perturbations; but a well chosen finite number of its periodic points and the periodic orbits that shadow them may be turned permanent by a $C^0$-small perturbation.

\smallskip

The first main difficulty of this argument is to adjust the size of the needed $C^0$ perturbations with the separation rates of the strips in the horseshoe, in order to be able to apply the combinatorial approach of \cite[Theorem 1.6]{BB}.
The second difficulty is to ensure that the ergodic probability measures supported on all of these orbits are distinct and sufficiently separated in the $W_p$ metric.

\smallskip

Let us briefly recall the reasoning to prove \cite[Proposition~7.1]{CRV} (we refer the reader to this reference for more details). Given $\delta >0$, a homeomorphism of $X$ can be arbitrarily $C^0$-approximated by another homeomorphism, say $f$, which has both a $q$-absorbing disk $B$ with diameter smaller than $\delta$, for some $q \in \mathbb{N}$, and a $C^0$-open neighborhood $\cW_f$ in $\mathrm{Homeo}(X,d)$ such that, for every $g \in \cW_f$, the disk $B$ is still $q$-absorbing for $g$. Then, by extra $q$ arbitrarily small $C^0$-perturbations, we get a homeomorphism $g \in \cW_f$ exhibiting a coherent $(\delta,\,\vep_k,\,q,\,\alpha)$-{pseudo-horseshoe} in the finite union $\hat B$ of the domains $\big(f^j(B)\big)_{0\,\leqslant\, j\, <\, q}$.

\smallskip

The previous construction is performed by an isotopy in $\hat B$, so we may assume that the homeomorphism $g$ is $C^1$-smooth on the open domain $\hat B$ and that there exists a subset $Q \subset \hat B$ such that the maximal invariant set $\Lambda = \bigcap_{n\,\in\, \mathbb Z} g^n(Q)$ is a horseshoe, $g^q$ has a horseshoe with
$$N = \Big\lfloor\Big(\frac1{\vep_k}\Big)^{\alpha\, \mathrm{dim}\, X}\Big\rfloor^q$$
strips and the Hausdorff distance between any two such strips is bounded below by $\vep^q_k$.
Let $T = T(\vep_k^q/4) \in \mathbb{N}$ given by the specification (cf. Subsection~\ref{sec:spec}) for the map $g$ restricted to $\Lambda$ and take an even positive integer $\ell$ (depending on $g$ and $\vep_k$) satisfying
\begin{equation}\label{eq:elle}
\quad q \ell \geqslant T
\end{equation}
and so that the diameter of each connected component of $\bigcap_{|n|\,\leqslant\, \ell/2} g^n(Q)$ is strictly smaller than $\vep_k^q$.

\smallskip

Using the methods of \cite{DF}, one can perform a finite number 
of arbitrarily small $C^0$-perturbations so that there is a $C^0$-open neighborhood of $g$ where a fixed finite number of periodic orbits become permanent, that is, persist under small $C^0$-perturbations of the dynamics. Therefore:

\begin{lemma}\label{le:specification}
There exists a $C^0$-open neighborhood $\widehat \cW_g \subset \cW_g$ of $g$ in $\mathrm{Homeo}(X,d)$ such that every $h \in \widehat \cW_g$  satisfies the following conditions:
\begin{itemize}
\item[(a)] $h$ has a coherent $(\delta,\,L\vep_k,\,q,\,\alpha)$-{pseudo-horseshoe}, for some $\delta>0$, $q \in \mathbb{N}$ and $L>0$.
\smallskip

\item[(b)]  $h$ has a collection $\mathcal F_h$ of $N = \lfloor\Big(\frac1{\vep_k}\Big)^{\alpha\, \mathrm{dim}\, X}\rfloor^q$ permanent periodic orbits of period $q$, and the Hausdorff distance between these orbits is bounded below by $\vep_k^q$.
\medskip

\item[(c)] $\bigcap_{n\,\in\, \mathbb Z} h^n(Q)$ is a pseudo-horseshoe and the diameter of each connected component of $\bigcap_{|n|\,\leqslant \,\ell/2} h^n(Q)$ is strictly smaller than $\vep_k^q$.
\medskip

\item[(d)] Assume that $N$ is an even integer (otherwise, replace $N$ by $2\lfloor N/2 \rfloor$). For any collection
$$\underline P \,=\, (P_1, \,P_2, \,\cdots, \,P_{\frac{N}{2}})$$
of $N/2$ periodic orbits in ${\mathcal F}_h$ there is a periodic orbit $\mathcal P = \mathcal P(\underline P)$ with period $q\ell N/2 + T N/2$ which $\vep_k^q/4$-shadows the pseudo-orbit
$$\Big(\underbrace{(P_1,\, P_1,\, \cdots,\, P_1)}_{\ell},\, \underbrace{(P_2,\, P_2,\, \dots, \,P_2)}_{\ell},\, \cdots, \,\underbrace{(P_{\frac{N}{2}},\, P_{\frac{N}{2}},\, \cdots,\, P_{\frac{N}{2}})}_{\ell}\Big)$$
with a time lag of $T$ iterates in between.
\end{itemize}
\end{lemma}

\smallskip

We define the subset $\widehat{\mathcal O}_K(\vep_k,\alpha)$ as the union of the previously obtained open domains $\widehat \cW_g$. By construction, this is a $C^0$-open subset of $\mathrm{Homeo}(X,d)$ and, given $K \in \mathbb{N}$, the union $\widehat{\mathcal O}_K(\alpha)\, = \,\bigcup_{\substack{k \,\, \in \,\, \mathbb{N} \\ k \,\, \geqslant\,\, K}}\, \widehat{\mathcal O}(\vep_k, \alpha)$ is $C^0$-dense in $\mathrm{Homeo}(X,d)$. We are left to prove that, if $h \in \widehat{\mathcal O}(\vep_k, \alpha)$, then there exists a subset $E_h(X) \subset \mathcal M^{\mathrm{erg}}_h(X)$ with the properties listed in Proposition~\ref{C0_generic1-extra}.
For that we will apply the combinatorial estimates used in the proof of \cite[Theorem~1.6]{BB}.

\smallskip

According to \cite{BB}, every maximal $\frac{N}4$-separated set in the space
$$F=\Big\{\beta \colon \,\{1,2, \dots, N\} \,\to \,\{0,1\} \text{ such that } \sum_{i=1}^N \beta(i) = \frac{N}2\Big\}$$
endowed with the Hamming metric, has cardinality bounded below by $D_1 e^{C_1 N}$, for some uniform constants $D_1, C_1 > 0 $. Given $h \in \widehat{\mathcal O}(\vep_k, \alpha)$, fix a $N/4$-maximal separated set $F'\subset F$ and consider the space $E_h(X)$ of ergodic probability measures defined by
$$\mu_{\beta} \,=\, \frac1{q\ell N/2 + T N/2}\,\sum_{j=0}^{q\ell N/2 + T N/2 - 1} \,\beta(i_j)\,\delta_{h^j(\mathcal P(\underline P_{\beta}))}$$
where $\beta \in F'$, $\underline P_{\beta}\,=\,(P_{i_1},\, P_{i_2}, \dots, P_{i_\frac{N}{2}})$ and $\beta(i_j) = 1$ for every $1 \leqslant j \leqslant N/2$. Note that the cardinality of $E_h(X)$ coincides with the one of $F'$.

\smallskip

We claim that any two probability measures in $E_h(X)$ are $8^{-\frac1p} \vep_k^{q}$-separated in the metric $W_p$. Firstly, observe that, if $\beta_1,\beta_2 \in F'$ are distinct, then
$$\underline P_{\beta_1} \,=\, (P_{i_1},\,P_{i_2}, \cdots, P_{i_j},\cdots, P_{i_\frac{N}{2}})\,\, \neq \,\, \underline P_{\beta_2}  \,=\, (P_{k_1},\,P_{k_2}, \cdots, P_{k_j}, \cdots P_{k_\frac{N}{2}})$$
and these two vectors differ in at least $N/4$ entries (that is, there are at least $N/4$ values of $1 \leqslant j \leqslant N/2$ such that $P_{i_j} \neq P_{k_j}$). Moreover, using item (d) of Lemma~\ref{le:specification}, we conclude that, for any such values of $j$, one has
$$d\Big(h^{t + (j-1) q \ell}(\mathcal P(\underline P_{\beta_1})),\,\, h^{t +(j-1) q \ell}(\mathcal P(\underline P_{\beta_2}))\Big) \,>\, \vep_k^q/2 \quad \quad \forall \, 0\leqslant t \leqslant q \ell.$$
Due to the choice of $\ell$ (see \eqref{eq:elle}), given $\pi\,\in\,\Pi(\mu_{\beta_1}, \mu_{\beta_2})$ one has
\begin{align*}
\int_{X\times X}\,[d(x,y)]^p\;d\pi(x,y) & = \int_{\supp (\mu_{\beta_1})\times \supp (\mu_{\beta_2})}\,[d(x,y)]^p\;d\pi(x,y) \\
& \smallskip \\
& \geqslant \frac{\vep_k^{pq}}2\; \pi\Big( \big\{(x,y)\in X\times X \colon d(x,y) \,>\, \frac{\vep_k^q}2\big\} \Big) \\
& \smallskip \\
& \geqslant \frac{\vep_k^{pq}}2\;  \frac{q\ell N/4}{q\ell N/2 + T N/2}\\
& \smallskip \\
& = \frac{\vep_k^{pq}}4\;  \frac{q\ell}{q\ell + T}\\
& \smallskip \\
& \geqslant \frac{\vep_k^{pq}}8.
\end{align*}
Thus
\begin{align*}
W_p(\mu_{\beta_1}, \mu_{\beta_2})\, =\, \inf_{\pi\,\in\,\Pi(\mu_{\beta_1}, \mu_{\beta_2})}\,\left(\int_{X\times X}\,[d(x,y)]^p\;d\pi(x,y)\right)^{1/p} \,\geqslant\, 8^{-\frac1p} \vep_k^{q}.
\end{align*}
This ends the proof of Proposition~\ref{C0_generic1-extra}.
\end{proof}

\smallskip

\subsection{Estimate of the topological emergence} It is immediate to deduce from Proposition~\ref{C0_generic1-extra} that the set
\begin{equation}\label{eq:set-R}
\mathfrak{R} \, = \, \bigcap_{\alpha\, \in \,\,]0,1[ \,\cap\, \mathbb Q} \; \bigcap_{K \,\in \,\mathbb{N}} \; \bigcup_{\substack{k \,\, \in \,\, \mathbb{N} \\ k \,\, \geqslant\,\, K}}\, \widehat{\mathcal O}(\vep_k, \alpha)
\end{equation}
is $C^0$-Baire residual in $\mathrm{Homeo}(X,d)$. We will show that
$$\overline{\mathrm{mo}}\,(\mathcal{M}^{\mathrm{erg}}_f(X),\,\mathrm W_p) \,=\, \mathrm{dim} X \quad \quad \forall\, f \in \mathfrak{R}.$$

\smallskip

The estimates in the previous subsection indicate that, given $K \in \mathbb{N}$, an integer $k \geqslant K$ and a rational number $\alpha\in \,\,]0,1[$, there are $q_k\ \in \mathbb{N}$ and a uniform constant $C'>0$ such that
$$\frac{\log\log S_{\mathcal{M}^{\mathrm{erg}}_f(X)}(8^{-\frac1p}\vep_{k}^{q_k})}{-\log \,(8^{-\frac1p}\vep_{k}^{q_k})} \,\geqslant\,
\frac{\log (C + \frac1C \lfloor\Big(\frac1{\vep_k}\Big)^{\alpha\, \mathrm{dim}\, X}\rfloor^{q_k})}{- {q_k}\log \vep_{k} -\log (8^{-\frac1p})}
\,\geqslant\, \frac{\log (C') - q_k\,\log \,(\vep_k)\, \alpha\, \mathrm{dim}\, X }{-{q_k}\log \vep_{k} - \log (8^{-\frac1p})}.$$
Consequently, taking the $\limsup$ as $k$ goes to $+ \infty$, we get
$$\overline{\mathrm{mo}}\,(\mathcal{M}^{\mathrm{erg}}_f(X),\,\mathrm W_p)\,\geqslant\, \alpha\, \mathrm{dim}\, X.$$
As $\alpha$ may be chosen arbitrarily close to $1$, we conclude that $\overline{\mathrm{mo}}\,(\mathcal{M}^{\mathrm{erg}}_f(X),\,\mathrm W_p) = \mathrm{dim}\, X,$ as claimed. The proof of Theorem~\ref{thm:main1} is complete.

\smallskip

\subsection{Pseudo-physical measures}

Assume that the manifold $X$ is endowed with a volume reference measure, which we call Lebesgue measure. Given $\mu \in \mathcal M_f(X)$, denote by $\mathcal{L}_\omega(x,f)$ the set of accumulation points in the weak$^*$-topology of the sequence $\big(\mathrm{e}^f_n(x)\big)_{n \,\in\, \mathbb N}$ of $n^{\text{th}}$-empirical measures associated to $x$ by $f$. The measure $\mu$ is called \emph{physical} if the set of those $x \in X$ for which $\mathcal{L}_\omega(x,f) = \{\mu\}$ has positive Lebesgue measure in $X$. Recall from \cite{CE} that $\mu \in \mathcal M_f(X)$ is said to be \emph{pseudo-physical} if, for every $\vep > 0$, the set
$$A_\vep(\mu) \,=\, \big\{x \in X \colon \,\mathrm{dist}(\mu, \nu) \,<\, \vep \quad \forall\, \nu \in \,\mathcal{L}_\omega(x, f)\big\}$$
has positive Lebesgue measure, where $\mathrm{dist}$ stands for any distance inducing in $\mathcal M(X)$ the weak$^*$-topology.

\smallskip

\begin{remark}
It was proved in \cite[Proposition 1.2]{KNS22} that, if $X$ is an infinite compact metric space and $f \colon X \to X$ is a continuous map with the specification property, then there is a residual subset $Y$ of $X$ such that
$$\mathrm{\overline{dim}_B}\, \big(\mathcal{L}_\omega(x,f)\big) = +\infty \quad \quad \forall \, x \in Y.$$
When $\mu$ is a $\mathbb{O}\mathbb{U}$-probability measure, as the specification property is $C^0$-generic in $\mathrm{Homeo}_\mu(X,d)$ so is the previous equality.
\end{remark}

\smallskip

Let $\mathcal O_f(X)$ be the set of pseudo-physical measures of $f$ and $\mathcal M^{\mathrm{per}}_f(X)$ be the set of periodic Dirac measures of $f$. It is known (cf. \cite[Theorem 1]{CT}) that, for a $C^0$-generic $f$ in $\mathrm{Homeo}(X,d)$, one has
$$\overline{\mathcal M^{\mathrm{erg}}_f}(X) \,=\, \overline{\mathcal M^{\mathrm{per}}_f}(X) \,=\, \mathcal O_f(X)$$
where the closures are taken in the weak$^*$-topology. Moreover, for a $C^0$-generic $f$ in $\mathrm{Homeo}(X,d)$, the set $\mathcal O_f(X)$ has empty interior in $\mathcal M_f(X)$, so $\mathcal M_f(X) \setminus \mathcal O_f(X)$ is an open dense subset of $\mathcal M_f(X)$ which does not intersect $\overline{\mathcal M^{\mathrm{erg}}_f}(X)$ (cf. \cite[Theorem 2]{CT}). Therefore, in spite of $\mathcal O_f(X)$ being meager,
$$\overline{\mathrm{mo}}\,(\mathcal M^{\mathrm{erg}}_f(X), \, \mathrm W_p) \,= \,\overline{\mathrm{mo}}\,(\mathcal O_f(X), \, \mathrm W_p)$$
and similarly regarding the metric $\mathrm{LP}$. Hence, from Theorem~\ref{thm:main1} we conclude that:

\begin{corollary} For a $C^0$-generic $f$ in $\mathrm{Homeo}(X,d)$ one has
$$\overline{\mathrm{mo}}\,(\mathcal O_f(X), \, \mathrm W_p) \,=\, \mathrm{\dim} X.$$
\end{corollary}

\section{Proof of Theorem~\ref{thm2}}

The content of this section is inspired by the intermediate value property of the upper box dimension of bounded subsets of the Euclidean space $\mathbb R^\ell$, $\ell \in \mathbb{N}$, proved in \cite[Theorem~2]{Feng-Zhiying}.
Let $(Z,d)$ be a compact metric space and fix an arbitrary $0 \leqslant \beta \leqslant \mathrm{\overline{mo}}\,(Z)$. If $\beta \in \{0, \,\mathrm{\overline{mo}}\,(Z)\}$, it is immediate to find a subset $Y_\beta\subset Z$ such that $\mathrm{\overline{mo}}\,(Y_\beta)=\beta$.

\smallskip

Now we will consider $\beta \in \,\,]0, \,\mathrm{\overline{mo}}\,(Z)[$.
We start by showing that in order to evaluate the upper metric order of $Y$
$$ \overline{\mathrm{mo}}\,(Y) \, = \,  \limsup_{\varepsilon\,\to\, 0^+}\,\frac{\log\log S_Y(\varepsilon)}{-\log \varepsilon}$$
we may use balls of radius $\lambda^j$, for $j \in \mathbb{N}$, and any choice of $0 < \lambda < 1$. More precisely:

\begin{lemma}\label{le:sequencia} Given $\lambda \in \,\,]0,1[$, for every subset $Y$ of $Z$ one has
$$\mathrm{\overline{mo}}\,(Y) \,= \,\limsup_{j\,\to\,+\infty} \,\frac{\log\log S_Y(\lambda^{j})}{-\log \lambda^{j}}.$$
\end{lemma}

\begin{proof} Given $\lambda \in \,\,]0,1[$ and $\vep > 0$, there is a positive integer $j$ such that $\lambda^{j+1} < \vep \leqslant \lambda^{j}$. Then
$$\frac{1}{-\log \lambda - \log(\lambda^j)} \,<\,\frac{1}{-\log \vep} \,\leqslant \, \frac{1}{\log \lambda - \log(\lambda^{j+1})}$$
and
$S_Y(\lambda^{j}) \,\leqslant \, S_Y(\vep) \, \leqslant \, S_Y(\lambda^{j+1})$
which imply that
$$
\frac{\log\log S_Y(\lambda^{j+1})}{-\log \lambda^{j+1}} \geqslant \frac{\log\log S_Y(\varepsilon)}{-\log \varepsilon}
	\geqslant \frac{\log\log S_Y(\lambda^{j})}{-\log \lambda^{j}}.
$$
Consequently,
\begin{eqnarray*}
 \limsup_{\varepsilon\,\to\, 0^+}\,\frac{\log\log S_Y(\varepsilon)}{-\log \varepsilon}
 	= \limsup_{j\,\to\,+\infty}\, \frac{\log\log S_Y(\lambda^{j})}{-\log \lambda^{j}}.
\end{eqnarray*}
\end{proof}

\smallskip

Let us resume the proof of the theorem when $0 < \beta < \overline{\mathrm{mo}}\,(Z)$. We start by choosing $\lambda \in \,\,]0,\frac12[$. By compactness of $Z$, for each $k \in \mathbb{N}$ there is a finite open covering $\mathcal U_k$ of $Z$ by balls of radius $\lambda^k$ whose corresponding balls of radius $\lambda^{k+1}$ are pairwise disjoint. In particular, there exists a partition $\mathcal P_k$ of $Z$ made up of elements whose diameter is bounded by $\lambda^k$ and whose inner diameter is bounded below by $\lambda^{k+1}$.

\smallskip

As $\mathrm{\overline{mo}}\,(Z) > \beta$, there are infinitely many positive integers $k$ such that $S_Z(\lambda^{k-1})$ is bigger than
$\lfloor \exp(\lambda^{-\beta {k}}) \rfloor.$ Let $k_1 \in \mathbb{N}$ be the smallest of them, which satisfies
$$S_Z(\lambda^{k_1-1}) >  \lfloor \exp(\lambda^{-\beta {k_1}}) \rfloor.$$
As the diameter of the elements of the partition $\mathcal P_{k_1}$ is smaller than $\lambda^{k_1}$ and $0<\lambda<\frac12$,
we are sure that any two $\lambda^{k_1-1}$-separated points belong to different elements of the partition $\mathcal P_{k_1}$.
Therefore, there exist at least $\lfloor \exp(\lambda^{-\beta {k_1}}) \rfloor$ elements of the partition $\cP_{k_1}$ which intersect $Z$.
Moreover, since the upper metric order is finitely stable (cf. \cite{BB}), that is,
$$\mathrm{\overline{mo}}\,(\bigcup_{1\,\leqslant\, j\, \leqslant\, n} B_j) \,= \,\max_{1\,\leqslant\, j\, \leqslant\, n} \overline{\mathrm{mo}}\,(B_j)$$
for any collection $\{B_j\}_{1\,\leqslant\, j \,\leqslant\, n}$ of subsets of $Z$, there exists a partition element $E_{k_1} \in \mathcal P_{k_1}$ such that
\begin{equation}\label{mo1}
\overline{\mathrm{mo}}\,(Z\cap E_{k_1})\,=\,\overline{\mathrm{mo}}\,(Z).
\end{equation}
Select a finite sample of points
$$\widehat Y_{k_1} \,=\,\Big\{x_{1,i} \colon \,1\leqslant i \leqslant \lfloor \exp(\lambda^{-\beta k_1}) \rfloor\Big\}\subset Z$$
which belong to different elements of the partition $\mathcal P_{k_1} \setminus \{E_{k_1}\}$. Afterwards, take the set
$$Y_{1} \,=\, \widehat Y_{k_1} \; \cup \; (Z \cap E_{k_1}).$$
By construction, the equality~\eqref{mo1} and the finite stability of the upper metric order, one has
\begin{enumerate}
\item  $\quad \mathrm{\overline{mo}}\,(Y_{1})\,=\,\overline{\mathrm{mo}}\,(Z)$;
\item  $\quad \#\{E\in \cP_{k_1}\colon \, E\cap Y_{1}\}\,=\,\lfloor \exp(\lambda^{-\beta k_1}) \rfloor+1$;
\item  $\quad \#\{E\in \cP_{k}\colon \,E\cap Y_1\}\,\leqslant \,\lfloor \exp(\lambda^{-\beta k}) \rfloor$ for every $1 \leqslant k < k_1$.
\end{enumerate}
\smallskip

\noindent By item (1), one can take the smallest integer $k_2 > k_1$ such that
$$S_{Y_1}(\lambda^{k_2-1}) >  \lfloor \exp(\lambda^{-\beta {k_2}} \rfloor$$
and so there are at least $\lfloor \exp(\lambda^{-\beta {k_2}}) \rfloor$ elements of the partition $\cP_{k_2}$ which intersect $Y_1$.
Thus, there exists $E_{k_2}\in \cP_{k_2}$ such that
$$\overline{\mathrm{mo}}\,(Z\cap E_{k_1})\,=\,\overline{\mathrm{mo}}\,(Z).$$
Again, take a finite collection of points
$$\widehat Y_{k_2}\,=\, \Big\{x_{2,i} \colon\, 1\leqslant i \leqslant \lfloor \exp(\lambda^{-\beta k_2}) \rfloor\Big\} \subset Y_1$$
belonging to different elements of the partition $\mathcal P_{k_2} \setminus \{E_{k_2}\}$. Then consider the set
$$Y_{2} \,=\, \widehat Y_{1} \; \cup \; (Z \cap E_{k_2})$$
which satisfies
\begin{enumerate}
\item[(4)] $\quad \mathrm{\overline{mo}}\,(Y_{2})\,=\,\overline{\mathrm{mo}}\,(Z)$;
\item[(5)] $\quad \#\{E\in \cP_{k_2}\colon \,E\cap Y_{2}\}\, =\, \lfloor \exp(\lambda^{-\beta k_2}) \rfloor+1$;
\item[(6)] $\quad \#\{E\in \cP_{k}\colon \,E\cap Y_{2}\}\,\leqslant\, \lfloor \exp(\lambda^{-\beta k}) \rfloor$ for every $k_1 < k < k_2$.
\end{enumerate}
\smallskip

\noindent Proceeding recursively, one constructs a nested sequence of sets
$Y_n \subset \dots \subset Y_2 \subset Y_1 \subset Z$
whose upper metric orders coincide with $\overline{\mathrm{mo}}\,(Z)$, and such that
\begin{equation}\label{eq:dim1}
\#\{E\in \cP_{k_n}\colon \, E\cap Y_{n}\}\,=\,\lfloor \exp(\lambda^{-\beta k_n}) \rfloor + 1
\end{equation}
and
\begin{equation}\label{eq:dim2}
\#\{E\in \cP_{k}\colon \, E\cap Y_{n}\}\,\leqslant\, \lfloor \exp(\lambda^{-\beta k}) \rfloor \quad \text{ for every $k_{n-1} < k < k_n$}.
\end{equation}
\smallskip

\noindent In particular, bringing together equations ~\eqref{eq:dim1} and ~\eqref{eq:dim2}, and the fact that the inner diameter of $\mathcal P_k$
is bounded below by $\lambda^{k+1}$, we conclude that the subset of $Z$ defined by
$$Y_\beta \,=\, \bigcap_{n\,\in \, \mathbb{N}}\, Y_n$$
has upper metric order $\mathrm{\overline{mo}}\,(Y_{\beta}) = \beta$. The proof of the theorem is complete.

\section{Proof of Theorem~\ref{cor:1}}

Let $(X,d)$ be a compact metric space. Consider a continuous map $f: X\to X$ and take $\beta \in [0, \, \mathcal{E}_{\mathrm{top}}(f)]$. The following argument is inspired by the proof of \cite[Theorem E]{BB}, where the case $\beta = \mathcal{E}_{\mathrm{top}}(f)$ was addressed.

\smallskip

Assume that $\mathcal{E}_{\mathrm{top}}(f) > 0$ and fix $\beta \in [0, \, \mathcal{E}_{\mathrm{top}}(f)[$. By Theorem~\ref{thm2} applied to $Z = \mathcal{M}^{\mathrm{erg}}_f(X)$, whose upper metric order $\mathrm{\overline{mo}}\,(Z)$ is precisely $\mathcal{E}_{\mathrm{top}}(f)$, there exists a subset $Y_\beta \subset Z$ such that $\mathrm{\overline{mo}}\,(Y_\beta) = \beta$. Therefore, by \cite[Theorem 3.9]{BB} we may find a probability measure $\nu \in {\mathcal M}_1(\mathcal{M}^{\mathrm{erg}}_f(X))$ such that $\overline{q_0}(\nu) = \mathrm{\overline{mo}}\,(Y_\beta)$, where $\overline{q_0}$ stands for the quantization of $\nu$. (We refer the reader to \cite{GraLus} for more details regarding this notion, which aims at approximating
$\nu$, in the Wasserstein or $\mathrm{LP}$ metric, by measures with finite support.) Then the probability measure $\mu = \int_{\mathcal{M}_1(X)}\,\eta \, d\nu(\eta)$ is $f$-invariant, so we may apply \cite[Proposition 3.12]{BB} to $\mu$ and thus conclude that
$$\mathcal{E}_\mu(f) \,=\, \limsup_{\vep \, \to \, 0^+} \, \frac{\log \log \mathcal{E}_\mu(f)(\vep)}{-\log \vep} \,=\, \overline{q_0}(\nu) \,=\, \mathrm{\overline{mo}}\,(Y_\beta) \,=\, \beta.$$
This ends the proof of the theorem.



\bigskip

\subsection*{Acknowledgments}
The authors are grateful to J. Fraser and N. Jurga for bringing to their attention useful references on Dimension Theory. MC and PV were partially supported by CMUP, which is financed by national funds through FCT - Funda\c c\~ao para a Ci\^encia e a Tecnologia, I.P., under the project with reference UIDB/00144/2020, and also acknowledge financial support from the project PTDC/MAT-PUR/4048/2021. PV benefits from the grant CEECIND/03721/2017 of the Stimulus of Scientific Employment, Individual Support 2017 Call, awarded by FCT. This work was initiated during the visit of FR to CMUP, whose hospitality is gratefully acknowledged.


\begin{thebibliography}{99}

\bibitem{Akin}
E. Akin.
\newblock \emph{Stretching the Oxtoby-Ulam theorem}.
\newblock Colloq. Math. 84/85:1 (2000) 83--94.


\bibitem{Berger}
P. Berger.
\newblock \emph{Emergence and non-typicality of the finiteness of the attractors in many topologies}.
\newblock Proc. Steklov Inst. Math. 297:1 (2017) 1--27.

\bibitem{BB}
P. Berger and J. Bochi.
\newblock \emph{On emergence and complexity of ergodic decompositions}.
\newblock Adv. Math. 390 (2021) Paper No. 107904, 52 pp.

\bibitem{Bil}
P. Billingsley.
\newblock  Convergence of Probability Measures.
\newblock \emph{Wiley Series in Probability and Statistics}, John Wiley and Sons, New York, 2nd edition, 1999.

\bibitem{Bseminar}
J. Bochi.
\newblock \emph{Topological emergence}.
\newblock Talk at the Zoominar in Dynamical Systems, University of Porto, https://sites.google.com/view/dynsys-zoominar/home, 2021.

\bibitem{Bowen}
R. Bowen.
\newblock \emph{Periodic points and measures for Axiom {$A$} diffeomorphisms.}
\newblock Trans. Amer. Math. Soc. 154 (1971) 377--397.

\bibitem{CRV}
M. Carvalho, F. B. Rodrigues and P. Varandas.
\newblock \emph{Generic homeomorphisms have full metric mean dimension}.
\newblock Ergodic Theory Dynam. Systems 142 (2020) 1--25.


\bibitem{CE}
E. Catsigeras and H. Enrich.
\newblock \emph{SRB-like measures for $C^0$ dynamics}.
\newblock Acad. Sci. Math. 59:2 (2011) 151--164.

\bibitem{CT}
E. Catsigeras and S. Troubetzkoy.
\newblock \emph{Invariant measures for typical continuous maps on manifolds}.
\newblock Nonlinearity 32:10 (2019) 3981--4001.


\bibitem{DF}
F. Daalderop and R. Fokkink.
\newblock \emph{Chaotic homeomorphisms are generic.}
\newblock Topol. Appl. 102 (2000) 297--302.

\bibitem{DGS}
M. Denker, C. Grillenberger and K. Sigmund.
\newblock Ergodic Theory on Compact Spaces.
\newblock \emph{Lecture Notes in Mathematics} 527, Springer-Verlag, Berlin-New York, 1976.


\bibitem{Falconer}
K. Falconer.
\newblock Fractal Geometry.
\newblock \emph{Mathematical Foundations and Applications.}, John Wiley \& Sons, New York, 2nd edition, 2003.

\bibitem{Feng-Zhiying}
D. Feng, Z. Wen and J. Wu.
\newblock \emph{Some remarks on the box-counting dimensions.}
\newblock Progr. Natur. Sci. (English Ed.) 9:6 (1999) 409--415.



\bibitem{GraLus}
S. Graf and H. Luschgy.
\newblock Foundations of Quantization for Probability Distributions.
\newblock \emph{Lecture Notes in Mathematics} 1730, Springer-Verlag, Berlin, 2000.


\bibitem{Guih}
 P.-A. Guih\'{e}neuf.
\newblock Propri\'{e}t\'{e}s dynamiques g\'{e}n\'{e}riques des hom\'{e}omorphismes conservatifs.
\newblock \emph{Ensaios Matem\'{a}ticos}. Sociedade Brasileira de Matem\'{a}tica, Rio de Janeiro, 2012.

\bibitem{GL}
P.-A. Guih\'eneuf and T. Lefeuvre.
\newblock \emph{On the genericity of the shadowing property for conservative homeomorphisms.}
\newblock  Proc. Amer. Math. Soc. 146 (2018) 4225--4237.









\bibitem{Hurley1}
M.~Hurley.
\newblock \emph{Generic homeomorphisms have no smallest attractors.}
\newblock Proc. Amer. Math. Soc. 123:4 (1995) 1277-1280.


\bibitem{Hurley2}
M.~Hurley.
\newblock \emph{Properties of attractors of generic homeomorphisms.}
\newblock Ergodic Theory Dynam. Systems 16:6 (1996) 1297--1310.


\bibitem{KNS22}
S. Kiriki, Y. Nakano and T. Soma.
\newblock \emph{Emergence via non-existence of averages}.
\newblock Adv. in Math. 400 (2022) Paper No. 108254, 30 pp.




\bibitem{Kolmogorov-Tikhomirov}
A. N. Kolmogorov and V. M. Tihomirov.
\newblock\emph{$\epsilon$-Entropy and $\epsilon$-capacity of sets in functional spaces.}
\newblock Amer. Math. Soc. Transl. 2:17 (1961) 277--364.


\bibitem{KLP}
D. Kwietniak, M. Lacka and P. Oprocha.
\newblock \emph{A panorama of specification-like properties and their consequences.}
\newblock Dynamics and Numbers, Contemporary Mathematics 669, AMS, Providence, RI, 2016, 155--186.




\bibitem{LW2000}
E. Lindenstrauss and B. Weiss.
\newblock \emph{Mean topological dimension.}
\newblock Israel J. Math. 115 (2000) 1--24.




\bibitem{Mis}
M.~Misiurewicz.
\newblock \emph{Ergodic natural measures.}
\newblock In: Algebraic and Topological Dynamics, Contemp. Math. 385, AMS, Providence, RI, 2005, 1--6.



\bibitem{Nitecki}
Z. Nitecki.
\newblock Differentiable Dynamics: An Introduction to the Orbit Structure of Diffeomorphisms.
\newblock MIT Press, Cambridge, Massachussets, 1971.

\bibitem{O-U}
J. Oxtoby and S. Ulam.
\newblock \textit{Measure-preserving homeomorphisms and metrical transitivity}.
\newblock Ann. of Math. 42:2 (1941) 874--920.

\bibitem{PPSS}
J.~Palis, C.~Pugh, M.~Shub and M.~Sullivan.
\newblock \emph{Genericity theorems in topological dynamics.}
\newblock \emph{Dynamical Systems - Warwick 1974}, Lecture Notes in Math. 468, Springer-Verlag New York, 1975, 241--250.


\bibitem{Pesin}
Ya. Pesin.
\newblock Dimension Theory in Dynamical Systems: Contemporary Views and Applications.
\newblock \emph{Lectures in Mathematics}, Chicago Press, 1997.


\bibitem{Rees}
M. Rees.
\newblock \textit{A Minimal positive entropy homeomorphism of the 2-torus}.
\newblock J. London Math. Soc. (2) 23 (1981) 537--550.

\bibitem{St}
V. Strassen.
\newblock \textit{The existence of probability measures with given marginals.}
\newblock Ann. Math. Statist. 36 (1965) 423--439.

\bibitem{Sun}
P. Sun.
\newblock \textit{Equilibrium states of intermediate entropies.}
\newblock Dyn. Syst. 36:1 (2021) 69--78.





\bibitem{Vi}
C. Villani.
\newblock Topics in Optimal Transportation.
\newblock \emph{Graduate Studies in Mathematics} 58, AMS, Providence, RI, 2003.


\bibitem{Wa}
P. Walters.
\newblock An Introduction to Ergodic Theory.
\newblock  Springer-Verlag New York, 1982.

\bibitem{Yano}
K. Yano.
\newblock \emph{A remark on the topological entropy of homeomorphisms.}
\newblock Invent. Math. 59 (1980) 215--220.


\end{thebibliography}
\end{document}